\documentclass[a4paper,11pt,english,spanish,leqno]{amsart}
\usepackage[utf8x]{inputenc}
\usepackage[bitstream-charter]{mathdesign}
\usepackage{amsmath, xcolor,amsthm,epsfig,epstopdf,url,array}
\usepackage[colorlinks=true]{hyperref}
\usepackage[all]{xy}
\usepackage{tikz}
\tikzset{node distance=2cm, auto}
\usepackage{subcaption}
\usepackage[backrefs]{amsrefs}
\usepackage{pgf,tikz}
\usetikzlibrary{arrows}

\theoremstyle{plain}
\newtheorem{thm}{Theorem}[section]

\newtheorem{lem}[thm]{Lemma}

\newtheorem{prop}[thm]{Proposition}

\newtheorem{cor}[thm]{Corollary}

\newtheorem{obs}[thm]{Remark}

\theoremstyle{definition}
\newtheorem{defn}[thm]{Definition}

\theoremstyle{remark}

\theoremstyle{plain}

\def\X{\mathfrak{X}_{*}^{(N)}}

\def\S#1{\mathbb{S}^{#1}}
\def\Esp#1#2{\text{\rm E}_{#1}\left[#2\right]}

\def\E#1{\mathcal{E}_{#1}}

\begin{document}

\title{Spherical cap discrepancy of the Diamond ensemble}

\author{Uju\'e Etayo}

\date{\today{}}

\thanks{The author has been supported by the Austrian Science Fund FWF project F5503 (part of the Special Research Program (SFB) Quasi-Monte Carlo Methods: Theory and Applications).
}

\subjclass[2010]{31A15; 65D30; 65D32.}

\keywords{Area regular partition, spherical points, spherical cap discrepancy, covering radius, quasi-uniform points}

\address{5010 Institute of Analysis and Number Theory 8010 Graz, Kopernikusgasse 24/II}
\email{etayo@math.tugraz.at}

\begin{abstract}
In \cite{EB18_3} the authors present a family of points on the sphere $\mathbb{S}^{2}$ depending on many parameters called the Diamond ensemble.
In this paper we compute the spherical cap discrepancy of the Diamond ensemble as well as some other quantities.
We also define an  area regular partition on the sphere where each region contains exactly one point of the set.
For a concrete choice of parameters, we prove that the Diamond ensemble provides the best spherical cap discrepancy known until date for a deterministic family of points.

\end{abstract}

\maketitle
\tableofcontents



\section{Introduction and main results}

Sets of points on the sphere $\mathbb{S}^{2}$ that are somehow well-distributed have been broadly studied on the literature, see for example \cite{Brauchart2015293}, \cite{10.2307/117605} and \cite{MR1306011}.
We  use the expression \textit{family of points} to denote a sequence of configurations of points on the sphere $\mathbb{S}^{2}$, $(\omega_{N})_{N}$ where $N$ is the number of points of the configuration.
$N$ does not necessarily cover every integer number, but an infinite subsequence of them.
In order to ease the notation, we will use $\omega_{N}$ indistintigly for \textit{family of points} or \textit{set of points}, although the meaning should be clear for the context.

Let us consider a family of points $\omega_{N} \subset \mathbb{S}^{2}$ and let $\mu$ be the Lebesgue measure on the sphere $\mathbb{S}^{2}$.
We recall that a Borel set $C\subset \mathbb{S}^{2}$ is $\mu$-continuous if  $\mu(\delta C) = 0$ where $\delta C$ is the border of $C$.
Then we say that $\omega_{N}$ is asymptotically uniformly distributed if
\begin{equation*}
\lim\limits_{N\rightarrow\infty}
\frac{\mu\left( \mathbb{S}^{2} \right) }{N}
\sum_{j=1}^{N}
f(x_{j})
=
\int_{\mathbb{S}^{2}} f(x)
d\mu.
\end{equation*}
where, for a fix $N$, $\omega_{N} = \{ x_{1},x_{2},...,x_{N}\}$ and the equation is satisfied for all continous function $f:\mathbb{S}^{2} \rightarrow \mathbb{R}$.
This definition is equivalent to the statement
\begin{equation}\label{asym_unif_dist}
\lim\limits_{N\rightarrow\infty}
\frac{\omega_{N}\cap C}{N}
=
\frac{\mu(C)}{\mu\left( \mathbb{S}^{2} \right)}
\end{equation}
for all $\mu$-continuous set $C$.
Asymptotically uniformity is one of the main conditions that one may ask a family of points in order to have an even distribution.
This notion is described in a more general context in \cite[Capítulo 3]{KN12}.

In this article we work with the spherical distance in $\mathbb{S}^{2}$.
Let us highlight, however, that the spherical distance and the euclidean distance are equivalent for small quantities and though, for all results presented in here.
The separation distance of a set of points $\omega_{N}$ is given by
\begin{equation*}
\delta\left( \omega_{N} \right)
=
\min
\limits_{1\leq i,j \leq N}
|| x_{i}-x_{j} ||,
\end{equation*}
and a family of points $\omega_{N}$ is said to be well-separated if 
\begin{equation*}
\delta\left( \omega_{N} \right)
\geq
\frac{c}{\sqrt{N}}
\end{equation*}
for some constant $c$ not depending on $N$.
The covering radius of a set of points on $\mathbb{S}^{2}$, also known as mesh norm, is defined as
\begin{equation*}
\rho\left( \omega_{N} \right)
=
\max
\limits_{y \in \mathbb{S}^{2}}
\min
\limits_{1\leq j \leq N}
|| y-x_{j} ||.
\end{equation*}
A family of points $\omega_{N}$ is a good covering if 
\begin{equation*}
\rho\left( \omega_{N} \right)
\leq
\frac{c}{\sqrt{N}}
\end{equation*}
for some constant $c$ not depending on $N$.
The relation between the minimal distance among points and the covering radius is usually refer as the mesh-separation ratio
\begin{equation*}
\gamma\left( \omega_{N} \right)
=
\frac{\rho\left( \omega_{N} \right)}{\delta\left( \omega_{N} \right)}
\end{equation*}
and can be thought as a condition number for approximation problems on the sphere.


\subsection{Spherical cap discrepancy}

Whenever we have a family of points that are asymptotically uniformly distributed, i.e. they converge towards the uniform distribution, we may ask what is the speed of convergence.
From formula \eqref{asym_unif_dist} we know that a family of points is asymptotically uniformly distributed if
\begin{equation*}
\lim\limits_{N\rightarrow\infty}
\frac{\omega_{N}\cap C}{N}
=
\frac{\mu(C)}{\mu\left( \mathbb{S}^{2} \right)}
\end{equation*}
for all $\mu$-continuous Borel set $C\subset \mathbb{S}^{2}$.
So, we want to study the quantity
\begin{equation*}
\lim\limits_{N\rightarrow\infty}
\left|
\frac{\omega_{N}\cap C}{N}
-
\frac{\mu(C)}{\mu\left( \mathbb{S}^{2} \right)}
\right|
\end{equation*}
that is called discrepancy.
The most classical discrepancy on the sphere is the so called spherical cap discrepancy, where we consider the set of all the spherical caps and the norm is whether the supremum or the $L^2$ norm.
We denote by $\text{cap}$ the set of all possible spherical caps in $\mathbb{S}^2$. 
Then we define the spherical cap discrepancy of a set of points $\omega_{N}$ as
\begin{equation}\label{eq_def_Discr}
 D_{\text{sup},\text{cap}} (\omega_{N}) 
=
\sup\limits_{C \in \text{cap}}
\left|
\frac{\omega_{N} \cap C}{N}
-
\frac{\mu(C)}{\mu\left( \mathbb{S}^{2} \right)}
\right|.
\end{equation}
If instead of the norm $\text{sup}$ we consider the norm $L^{2}$, then we can define the $L^{2}$ spherical cap discrepancy as
\begin{equation}\label{eq_def_Discr_L2}
 D_{L^{2},\text{cap}} (\omega_{N}) 
=
\int_{C \in \text{cap}}
\left|
\frac{\omega_{N} \cap C}{N}
-
\frac{\mu(C)}{\mu\left( \mathbb{S}^{2} \right)}
\right|.
\end{equation}
The Stolarsky invariance formula, stated in \cite{10.2307/2039137}, stablishs a relation between the sum of distances of the points from $\omega_{N}$ and the $\textrm{L}^{2}$ spherical cap discrepancy.
\begin{equation*}
c_{d}
\left(
D_{\textrm{L}^{2},cap} (\omega_{N})
\right)^{2}
=
\int_{\mathbb{S}^{d}}
\int_{\mathbb{S}^{d}}
||x-y||d\mu_{\mathbb{S}^{d}}(x)d\mu_{\mathbb{S}^{d}}(y)
-
\frac{1}{N^{2}}
\sum_{i,j=1}^{N} ||x_{i} - x_{j}||,
\end{equation*}
where $c_{d}$ is a constant depending only on the dimension of the sphere.
See also \cite{BD13} or \cite{BDM18} for more modern proofs ot the Stolarsky invariance formula.

\subsubsection{Minimal spherical cap discrepancy}
In \cite{beck_1984} it was shown that there exists a constant $c>0$, independent of $N$ such that for any $N$-point set $\omega_{N} \subset\mathbb{S}^2$ we have
\begin{equation*}
D_{\text{sup},\text{cap}}\left( \omega_{N}\right)
\geq
c 
N^{\frac{-3}{4}}.
\end{equation*}
On the other hand, using probabilistic methods it has been shown in \cite{Beck1984} that for all $N\geq 1$ there exist a point set $\omega_N$ in $\mathbb{S}^2$ satisfying
\begin{equation*}
D_{\text{sup},\text{cap}}\left( \omega_{N}\right)
\leq
c 
N^{\frac{-3}{4}}\log(N).
\end{equation*}
The proof of the last result is non-constructive.


\subsubsection{Probabilistic sets of points}

The spherical cap discrepancy of a random set of points coming from the uniform distribution on the sphere is of the order $N^{\frac{-1}{2}}$, see \cite{Aistleitner2012} for a proof.
In papers \cite{EJP3733} and \cite{BMOC2015energy}, authors define two determinantal point processes on the sphere $\mathbb{S}^{2}$ and they compute the spherical cap discrepancy obtaining
\begin{equation*}
D_{\text{sup},\text{cap}} \left( \omega_{N} \sim \X \right)
=
O\left( 
N^{\frac{-3}{4}}\log (N)
\right)
\end{equation*}
with overwhelming probability for the spherical ensemble (see \cite[Theorem 1.1]{EJP3733}) and 
the same for the harmonic ensemble, see \cite[Corollary 5]{BMOC2015energy}.
Here, $\omega_{N} \sim \X$ means a random set of $N$ different points on $\mathbb{S}^{2}$ following the distribution given by $\X$, the determinantal point process.


\subsubsection{Deterministic sets of points}

It is unknown how to construct a family of points with spherical cap discrepancy decaying as $N^{\frac{-3}{4}}\log (N)$.
The best bound given until date for a deterministic family of points can be found in the article \cite{Aistleitner2012}, where authors are able to bound the spherical cap discrepancy of the Fibonacci nodes by 
\begin{equation}\label{eq_Fibonaci}
D_{\text{sup},\text{cap}} \left( \omega_{N} \right)
\leq
44\sqrt{8}N^{\frac{-1}{2}}.
\end{equation}


\subsubsection{Riesz potentials and spherical cap discrepancy}
Given $s\in(0,\infty)$, the Riesz potential or $s$--energy of a set on points $\omega_{N} = \{ x_{1}, \ldots  ,x_{N} \}$ on the sphere $\mathbb{S}^{2}$ is
\begin{equation*}
\E{s}(\omega_{N}) = \displaystyle\sum_{i \neq j} \frac{1}{\| x_{i} - x_{j} \|^{s}}.
\end{equation*} 
This energy has a physical interpretation for some particular values of $s$, i.e. for $s=1$ the Riesz energy is the Coulomb potential and for $s=0$ the energy is defined by 
\begin{equation*}
\E{\log}(\omega_{N})=\left. \frac{d}{ds} \right|_{s=0}\E{s}(\omega_N) = \sum_{i \neq j} \log\| x_{i} - x_{j} \|^{-1} .
\end{equation*}
Finding quasiminimizers of the logarithmic energy is stated as the problem number 7 in the list of problems for the 21st century proposed by S. Smale, see \cite{Smale1998}.

There exist several results relating minimizers of spherical cap discrepancy and minimizers of Riesz energy. 
For example, minimizers of Riesz and logarithmic energy exhibit small spherical cap discrepancy, we refer to \cite{Brauchart2015293} and cites therein. 
The last word in this respect was given by Marzo and Mas who proved that any set of points that minimizes some Riesz energy with parameter $0\leq s< 2$ has spherical cap discrepancy bounded by
\begin{equation*}
 D_{\text{sup},\text{cap}} (\omega_{N}) 
 \leq
 c_{s}
N^{-\frac{2-s}{6-s}},
\end{equation*}
where $c_{s}$ is a constant depending only on $s$.
See \cite[Theorem 1.1]{MM2019} for the statement of the result in this full generality.
On the other hand, G\"otz proved that
\begin{equation*}
 D_{\text{sup},\text{cap}} (\omega_{N}) 
 \geq
 c
N^{-1/2},
\end{equation*}
with $c$ a constant not depending on $N$ for every family $\omega_{N}$ of minimizers of the logarithmic energy, see \cite[Corollary 2]{Gotz2000}.

From Theorem 1.1, Theorem 4.1 and Theorem 4.2 in \cite{EB18_3}, we know that the expectation value of the logarithmic energy associated to a particular choice of parameters of the Diamond ensemble is quite close to the minimal value; actually, for an especific choice of parameters it provides the lowest logarithmic energy known until date.
This fact suggests that the discrepancy of the Diamond ensemble should be of the order $N^{-1/2}$.
\begin{obs}\label{cor_sep}
The separation distance of the of minimal logarithmic energy configurations on $\mathbb{S}^2$ has been proved to be of the good order:
there exists a constant $c$ such that the distance in between any pair of points from a concret configuration is greater than $cN^{-1/2}$, for explicit values of the constant, we refer to \cite{Dragnev} and \cite{MR1306011}.
Since the logarithmic energy of the points coming from the Diamond ensemble is close to the minimal, their separation is expected to be of the right order.
\end{obs}

\subsection{Main results}\label{sub_Main_results}

In \cite{EB18_3}, authors present a constructive family of points defined by: the North pole, the South pole and sets of equispaced points located on several parallels. 
Is a parametrical model depending on the parallels chosen, the number of points chosen in each one and the rotation angle of every parallel.
The family is called the Diamond ensemble and it is denoted by $\diamond\left( N \right)$ where $N$ is the number of points.
This model is defined in full generality in section \ref{sub_Diamond}.

\begin{thm}\label{thm_main}
For any choice of parameters of the Diamond ensemble 
there exist two constants $c_{1}, c_{2} \in \mathbb{R}_{+}$ depending only on the parameters such that
\begin{equation*}
\frac{c_{1}}{\sqrt{N}}
\leq
D_{\text{sup},\text{cap}}\left(\diamond (N)\right)
\leq
\frac{c_{2}}{\sqrt{N}}.
\end{equation*}
\end{thm}

\begin{cor}\label{cor_L2}
For any choice of parameters of the Diamond ensemble 
we have
\begin{equation*}
D_{L^2,\text{cap}}\left(\diamond (N)\right)
\leq
\frac{c_{2}}{\sqrt{N}}
\end{equation*}
where $c_{2}\in\mathbb{R}$ is a fix constant that depends on the concrete model.
\end{cor}
Intuitively speaking, we tend to think that the $L^{2}$ spherical cap discrepancy of a set of points coming from the Diamond ensemble is lower than the bound proposed in Corollary \ref{cor_L2}.
There are $\sqrt{N}$ caps that present greater spherical cap discrepancy and that is where the sup spherical cap discrepancy arises, but they should not influence that much when we average over all spherical caps.

The constants $c_{1}$ and $c_{2}$ from Theorem \ref{thm_main} can be explicitly computed for any choice of parameters, and so, for the model presented in section \ref{sub_concrete} we have the following statement.

\begin{thm}\label{thm_concrete}
Let $\diamond (N)$ be the Diamond ensemble defined by $n=1$ and $r_j=4j$ for $1\leq j\leq M$. 
Then
\begin{equation*}
\frac{1}{\sqrt{N}}+ o\left( \frac{1}{\sqrt{N}} \right)
\leq
D_{\text{sup},\text{cap}}\left(\diamond (N)\right)
<
\frac{
4+ 2\sqrt{2}
}{\sqrt{N}}.
\end{equation*}
\end{thm}

Note that the choice of parameters in Theorem \ref{thm_concrete} is really simple and yet we obtain a bound for the discrepancy that is better than the best one known until date for a deterministic set of points, see formula \eqref{eq_Fibonaci}.
With better choices of parameters, for instance, with the ones proposed in \cite{EB18_3}, we shoud obtain better bounds.

If point coming from the Diamond ensemble are well separated as in Remark \ref{cor_sep}, then using Theorem \ref{thm_main} we could obtain a bound for the Riesz potential, we refer the reader to \cite[Theorem 5.4.1]{Leopardi}


\subsubsection{Proof of Theorem \ref{thm_main}}\label{sub_Proof_Main_result}

The proof of Theorem \ref{thm_main} follows from these two intermediate results.

\begin{thm}\label{thm_menor_que}
For any choice of parameters of the Diamond ensemble 
we have
\begin{equation*}
D_{\text{sup},\text{cap}}{\diamond (N)}
\leq
\frac{c_2}{\sqrt{N}}
\end{equation*}
where $c_2\in\mathbb{R}$ is a fix number that depends on the concrete model.
\end{thm}

The proof of Theorem \ref{thm_menor_que} follows the classical argument of Beck for the upper bound on the discrepancy, see for example \cite[Theorem 24D]{87e609866d654db2aa5e674048289af3}.
In order to prove it, we define an area regular partition on the sphere in section \ref{sec_ARP} and we complete the proof in section \ref{sec_Proof_menor_que}.

\begin{thm}\label{thm_mayor_que}
For any choice of parameters of the Diamond ensemble 
we have
\begin{equation*}
D_{\text{sup},\text{cap}}{\diamond (N)}
\geq
\frac{c_1}{\sqrt{N}}
\end{equation*}
where $c_1\in\mathbb{R}$ is a fix number that depends on the concrete model.
\end{thm}

For proving Theorem \ref{thm_mayor_que} it is enough to compute the value of  
\begin{equation*}
\left|
\frac{\diamond\left( N \right) \cap C}{N}
-
\frac{\mu(C)}{\mu\left( N \right)}
\right|
\end{equation*}
for a very precise spherical cap $C$, as we do in section \ref{sec_Proof_mayor_que}. 


\subsection{Organization of the paper}

In section \ref{sub_Diamond} we recall the principal characteristics of the Diamond ensemble presented in \cite{EB18_3} and prove some new results, esentially concerning the relation between the number of points on a given parallel and the total number of points.
In section \ref{sec_ARP} we present an area regular partition on the sphere coming from the Diamond ensemble and we prove some of its properties.
We employ the rest of the sections in proving the Theorem \ref{thm_concrete}, Theorem \ref{thm_menor_que} and Theorem \ref{thm_mayor_que}.


\section{The Diamond ensemble}\label{sub_Diamond}

\subsection{Definitions}

For this section we follow \cite{EB18_3}.
Fix $z\in(-1,1)$, the parallel of height $z$ in the sphere $\mathbb{S}^{2} \subset \mathbb{R}^{3}$ is simply the set of points $x\in\S2$ such that $\langle x,(0,0,1)\rangle=z$. 
Then we define a general construction of points as follows:
\begin{enumerate}
	\item Choose a positive integer $p$ and $z_1,\ldots ,z_p \in \mathbb{R}$ such that $1 > z_1> \ldots> z_p >-1$. Consider the $p$ parallels with heights $z_1,\ldots,z_p$.
	\item For each $j$, $1\leq j\leq p$, choose a number $r_j$ of points to be allocated on parallel $j$ (which is a circumference) by projecting the $r_j$ roots of unity onto the circumference and rotating them by a phase $\theta_j\in[0,2\pi]$, that also has to be chosen.
	\item To the already constructed collection of points, add the North and South pole.
\end{enumerate}
We denote this set by $\Omega(p,r_{j},z_{j},\theta_j)$. Explicit formulas for this construction are easily produced: points in parallel of height $z_j$ are of the form
\begin{equation*}
\begin{split}
& x = \left( \sqrt{1 - z_{j}^{2}} \cos \theta, \sqrt{1 - z_{j}^{2}} \sin \theta, z_{j} \right) \\
\end{split}
\end{equation*}
for some $\theta\in[0,2\pi]$ and thus the set we have described is defined by
	\begin{equation*}
	\Omega(p,r_{j},z_j, \theta_j) = 
	\begin{cases} 
	\mathcal{N} = (0,0,1) \\
	x_{j}^{i} = \left( \sqrt{1-z_{j}^{2}}\cos\left( \frac{2\pi i}{r_{j}} + \theta_{j} \right), \sqrt{1-z_{j}^{2}}\sin\left( \frac{2\pi i}{r_{j}}  + \theta_{j}\right), z_{j} \right) \\ 
	\mathcal{S} = (0,0,-1) 
	\end{cases}
	\end{equation*}
where $r_{j}$ is the number of roots of unity that we consider in the parallel $j$, $1 \leq j \leq p$ is the number of parallels, $1 \leq i \leq r_{j}$ and $0 \leq \theta_{j} < 2\pi$ is the rotation angle in parallel $j$.

We can rewrite $\Omega(p,r_{j},z_j,\theta_j)$ using spherical coordinates.
	\begin{equation*}
	\Omega(p,r_{j},z_j,\theta_j) = 
	\begin{cases} 
	\mathcal{N} = (0,0) \\
	x_{j}^{i} = \left(  \frac{2\pi i}{r_{j}} + \theta_{j}, \arctan\left( \frac{\sqrt{1 - z_{j}^{2}}}{z_{j}}  \right)\right) \\ 
	\mathcal{S} = (0,\pi) 
	\end{cases}
	\end{equation*}
where the first coordinate is an angle beewten $0$ and $2\pi$ defined in the plane $z=0$ and the second coordinate is an angle between $0$ and $\pi$ defined in the semiplane $x=0$, $y>0$.	
Note that since the point belong to the sphere, we don't write the coordinate correspondent to the radius, $r= 1$.
We obtain different families of points from the Diamond ensemble giving values to the parameters $p$, $\theta_j$, $r_{j}$ and $z_{j}$ as in the following definition.

\begin{defn}{\cite[Definition 3.1]{EB18_3}}\label{defn_Diamond}
Let $p,M$ be two positive integers with $p=2M-1$ odd and let $r_{j}=r(j)$ where $r:[0,2M]\to\mathbb{R}$ is a continuous piecewise linear function satisfying $r(x)=r(2M-x)$ and
\begin{equation*}
r(x)
=
   \begin{cases} 
      \alpha_1+\beta_1 x              & \mbox{if } 0=t_0 \leq x \leq t_1   \\
      \vdots&\vdots\\
\alpha_n+\beta_n x              & \mbox{if } t_{n-1} \leq x \leq t_n=M
   \end{cases}
\end{equation*}
Here, $[t_0,t_1,\ldots,t_n]$ is some partition of $[0,M]$ and all the $t_\ell, \alpha_\ell,\beta_\ell$ are assumed to be integer numbers.
The further assumptions on the parameters are that $\alpha_1=0$, $\alpha_\ell,\beta_\ell\geq0$, $\beta_1>0$ and there exists a constant $A\geq2$ not depending on $M$ such that $\alpha_\ell\leq AM$ and $\beta_\ell\leq A$ for all $1\leq \ell \leq n$. We also assume that $t_1\geq cM$ for some $c>0$. Moreover, let $z_{j}$ be as defined by the following proposition. 
\end{defn}

\begin{prop}{\cite[Proposition 2.5]{EB18_3}}\label{Propminparallel}
Given $\{ r_{1},...,r_{p} \}$ such that $r_{i} \in \mathbb{N}$, there exists a unique set of heights $\{z_{1},\ldots,z_{p} \}$ such that $z_{1} > \ldots > z_{p}$ and 
\begin{equation*}
\Esp{\theta_{1},...,\theta_{p} \in \left[0,2\pi\right]^{p}}{\E{\log}(\Omega(p,r_{j},z_j,\theta_{j}))}
\end{equation*}
 is minimized.
The heights are:
\begin{equation*}
z_{l}
=
\frac{\displaystyle\sum_{j=l+1}^{p}r_{j} - \displaystyle\sum_{j=1}^{l-1}r_{j}}{1 + \displaystyle\sum_{j=1}^{p}r_{j}}=1-\frac{1+r_l+2\sum_{j=1}^{l-1}r_j}{N-1},
\end{equation*}
where $N=2+\sum_{j=1}^{p}r_j$ is the total number of points. 
\end{prop}

\begin{obs}\label{rem_mono}
Note that since $\beta_1>0$ and we have $\alpha_{\ell} + \beta_{\ell}t_{\ell} = \alpha_{l+1} + \beta_{l+1}t_{\ell}$, the function $r(x)$ is monotonically increasing, in other words, $r_{j} \geq r_{k}$ if $j > k $.
\end{obs}

We call the family of points defined by the $r_{j}'s$ given in Definition \ref{defn_Diamond} and the $z_{j}'s$ as in Proposition \ref{Propminparallel} the \textit{Diamond ensemble} and we denote it by $\diamond (N)$, omiting in the notation the dependence on all the parameters $n$, $t_1,\ldots,t_n$, $\alpha_1,\ldots,\alpha_n$, $\beta_1,\ldots,\beta_n$, $\theta_1,\ldots,\theta_n$. 
We may not worry about the angle $\theta_j$, since the results here presented are valid for any choice of $\theta_j\in[0,2\pi ] $, so we denote $\Omega(p,r_{j},z_j,\theta_j)$ by $\Omega(p,r_{j},z_j)$.
The choice of parameters $n$ and $r_{\ell}$ for $1\leq \ell \leq n$ hence define a sequence of configurations of points were not all the integer numbers are taken but still the sequence goes to infinity as we make $M$ cover the natural numbers.

\begin{figure}[htp]
\centering
\includegraphics[width=.3\textwidth]{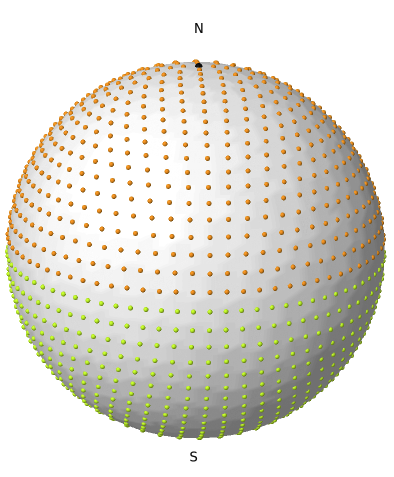}\hfill
\includegraphics[width=.35\textwidth]{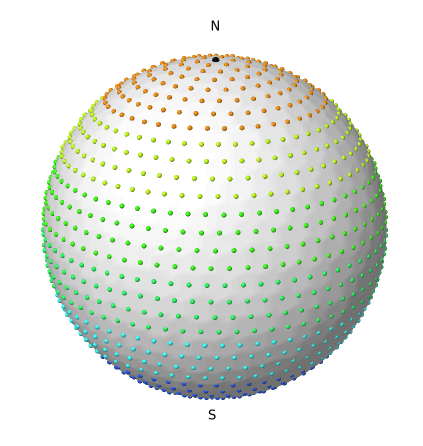}\hfill
\includegraphics[width=.3\textwidth]{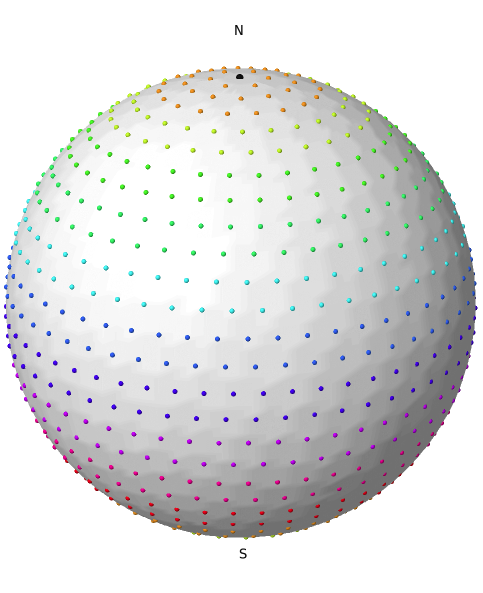}
\caption{Models of the Diamond ensemble for different choices of parameters.}
\label{fig:figure3}
\end{figure}

\subsection{Some extra properties}

The total number of points of $\diamond\left( N \right)$ is
\begin{equation*}
N=2-(\alpha_n+\beta_nM)+2\sum_{\ell=1}^n\;\sum_{j=t_{\ell-1}+1}^{t_\ell}(\alpha_\ell+\beta_\ell j).
\end{equation*}
We denote by 
\begin{equation}\label{eq_N_parcial}
N_j=1 + \sum_{k=1}^{j-1}r_k.
\end{equation}
Using the notation $N_{j}$ we can rewrite the value of $z_{j}$:
\begin{equation}\label{zjs}
z_j
=
1-\frac{1+r_j+2\sum_{k=1}^{j-1}r_k}{N-1}
=
1-\frac{2N_{j}}{N-1} - \frac{r_{j}-1}{N-1}.
\end{equation}
The following proposition shows the depende of $N$ on the number of parallels.
\begin{lem}\label{pos}
There exists constants $a_{1},a_{2}$ depending only on the choice of parameters $n, t_{\ell}, \alpha_{\ell}, \beta_{\ell}$ for all $1\leq \ell \leq n$ such that 
\begin{equation*}
a_{1}M^2 \leq N \leq a_{2}M^2
\end{equation*}
\end{lem}

\begin{proof}
From the properties of $\alpha_{\ell},\beta_{\ell}$ we have that
\begin{multline*}
N
=
2-(\alpha_n+\beta_nM)+2\sum_{\ell=1}^n\;\sum_{j=t_{\ell-1}+1}^{t_\ell}(\alpha_\ell+\beta_\ell j)
\leq
2+2\sum_{\ell=1}^n\;\sum_{j=t_{\ell-1}+1}^{t_\ell}(AM+A j)
\\
=
2+ 2AM^2 + AM(M+1)
=
3AM^2 + AM + 2,
\end{multline*}
where $A$ is the constant from Definition \ref{defn_Diamond}.
So it is enought to take $a_{2} = 4A$.
For the other inequality, using again the properties from Definition \ref{defn_Diamond} we have
\begin{equation*}
N_{t_{1}} 
= 
1+
\sum_{j=1}^{t_1-1}(\alpha_1+\beta_1 j)
\geq
1+
\sum_{j=1}^{cM-1} j
=
1+
\frac{cM(cM-1)}{2}
=
\frac{c^2}{2} M^2 - \frac{cM}{2} + 1.
\end{equation*}
We take $a_{1} = \frac{c^2-c}{2}$ and we conclude wiht
\begin{equation}\label{eq_N1}
N \geq N_{t_{1}} \geq a_{1}M^{2}.
\end{equation}

\end{proof}

\begin{lem}\label{lem_Nj}
There exist constants $k_{1},k_{2}\in\mathbb{R}_{+}$ depending only on the choice of parameters $n, t_{\ell}, \alpha_{\ell}, \beta_{\ell}$, $1\leq \ell \leq n$ such that for all $1 \leq j \leq M$ we have
\begin{equation*}
k_{1} r_{j}^2
\leq
N_{j}
\leq
k_{2} r_{j}^2
\end{equation*}
\end{lem}

\begin{proof}

For $1\leq j \leq t_{1}$ we have 
\begin{equation*}
N_{j} 
= 
1 + \sum_{k=1}^{j-1}(\alpha_1+\beta_1 k)
=
\frac{\beta_{1}}{2}j^2 + \left( \alpha_{1} - \frac{\beta_{1}}{2} \right) j + 1 - \alpha_{1}
\end{equation*}
and
\begin{equation*}
r_{j}^2
=
(\alpha_1+\beta_1 j)^2
=
\beta_{1}^2j^2 + 2\alpha_{1}\beta_{1}j + \alpha_{1}^2.
\end{equation*}
We observe that both $N_{j} $ and $r_{j}^2$ are positive branches of paraboles in $j$. 
Let us consider the functions
\begin{equation*}
r^2(x) = \beta_{1}^2x^2 + 2\alpha_{1}\beta_{1}x + \alpha_{1}^2 
\end{equation*}
and
\begin{equation*}
N(x)
= 
\frac{\beta_{1}}{2}x^2 + \left( \alpha_{1} - \frac{\beta_{1}}{2} \right) x + 1 - \alpha_{1}.
\end{equation*}
On the one hand, we have $r^2 (1) \geq 1 = N(1)$ and $(r^2(x))' \geq N(x)'$ for all $x\in(1,t_{1})$, so we can take $\dot{k_{2}} = 1$ and conclude that $N_{j} \leq \dot{k_{2}} r_{j}^2$ for all $1\leq j \leq t_{1}$.
On the other hand, if we take $\dot{k_{1}} = \frac{1}{2(\beta_{1}^2 + 2\alpha_{1}\beta_{1} + \alpha_{1}^2 )}$ then $\dot{k_{1}}r^{2}(1) = \frac{1}{2} < N(1)$ and $\dot{k_{1}}(r^2(x))' \leq N(x)'$ for all $x\in(1,t_{1})$ so we conclude that $\dot{k_{1}} r_{j}^2 \leq N_{j}$ for all $1\leq j \leq t_{1}$.

For $j > t_{1}$ we have that 
\begin{equation*}
r_{j}^{2} 
=
\left( \alpha_{\ell} + \beta_{\ell}j \right)^2 
\leq
\left( AM + AM \right)^2
=
 4A^2M^2
\end{equation*}
and from equation \eqref{eq_N1} we have that
\begin{equation*}
N_{j}
\geq
N_{t_{1}}
\geq
a_{1} M^2
\end{equation*}
So it is enough to take $\tilde{k_{1}} = \frac{a_{1}}{4A^2}$.
On the other had, by Lemma \ref{pos} we have 
\begin{equation*}
N_{j}
\leq
N
\leq
a_{2}M^2
\end{equation*}
and by the monotonicity of the function (see Remark \ref{rem_mono}) we have
\begin{equation*}
r_{j}^2 
=
\left(\alpha_{\ell} + \beta_{\ell}j \right)^2
\geq
t_{1}^{2}
\geq
c^2M^2.
\end{equation*}
So it is enough to take $\tilde{k_{2}} = \frac{a_{2}}{c^2}$.
We conclude by taking 
\begin{equation*}
k_{1} = \min\{ \tilde{k_{1}}, \dot{k_{1}}\}
\text{ and }
k_{2} = \min\{ \tilde{k_{2}}, \dot{k_{2}}\}
.
\end{equation*}

\end{proof}

\section{An area regular partition coming from the Diamond ensemble}\label{sec_ARP}


\subsection{About area-regular partitions on the sphere}

On the literature we can find several references to area regular partitions on the sphere $\mathbb{S}^{2}$ but no so many explicit examples of them, we refer to \cite{Alexander1972}, \cite{BL88}, \cite{doi:10.1002/rsa.10036} and \cite{10.2307/2039137}.
In \cite{Zhou} Zhou describes an area regular partition in $\mathbb{S}^{2}$ quite similar to the one that we present here.
The same construction is explained in \cite{MR1306011} and later in \cite{dolomites}.
This construction was modified by Bondarenko et al. \cite{Bondarenko2015} to create a partition with geodesic boundaries for the creation of well-separated spherical designs.
On his PhD dissertation \cite{Leopardi}, Leopardi studied the construction of Zhou generalizing it to higher dimensional spheres. 
He also provides a code in Mathlab available at \url{http://eqsp.sourceforge.net/} where one can obtain the area regular partition in $\mathbb{S}^{d}$ for any number of cells.


\subsection{ARP for the Diamond ensemble}

Given a family of points coming from the Diamond ensemble, we define an area regular partition by taking two spherical caps, one centered in the North pole and the other in the South pole and a collection of rectangular regions located in some collars, see Figure \ref{fig:figure1}.
As for the points, we allow a random angle of rotation $\theta_{j}$ in every collar.

\begin{figure}[htp]
\centering
\includegraphics[width=0.5\textwidth]{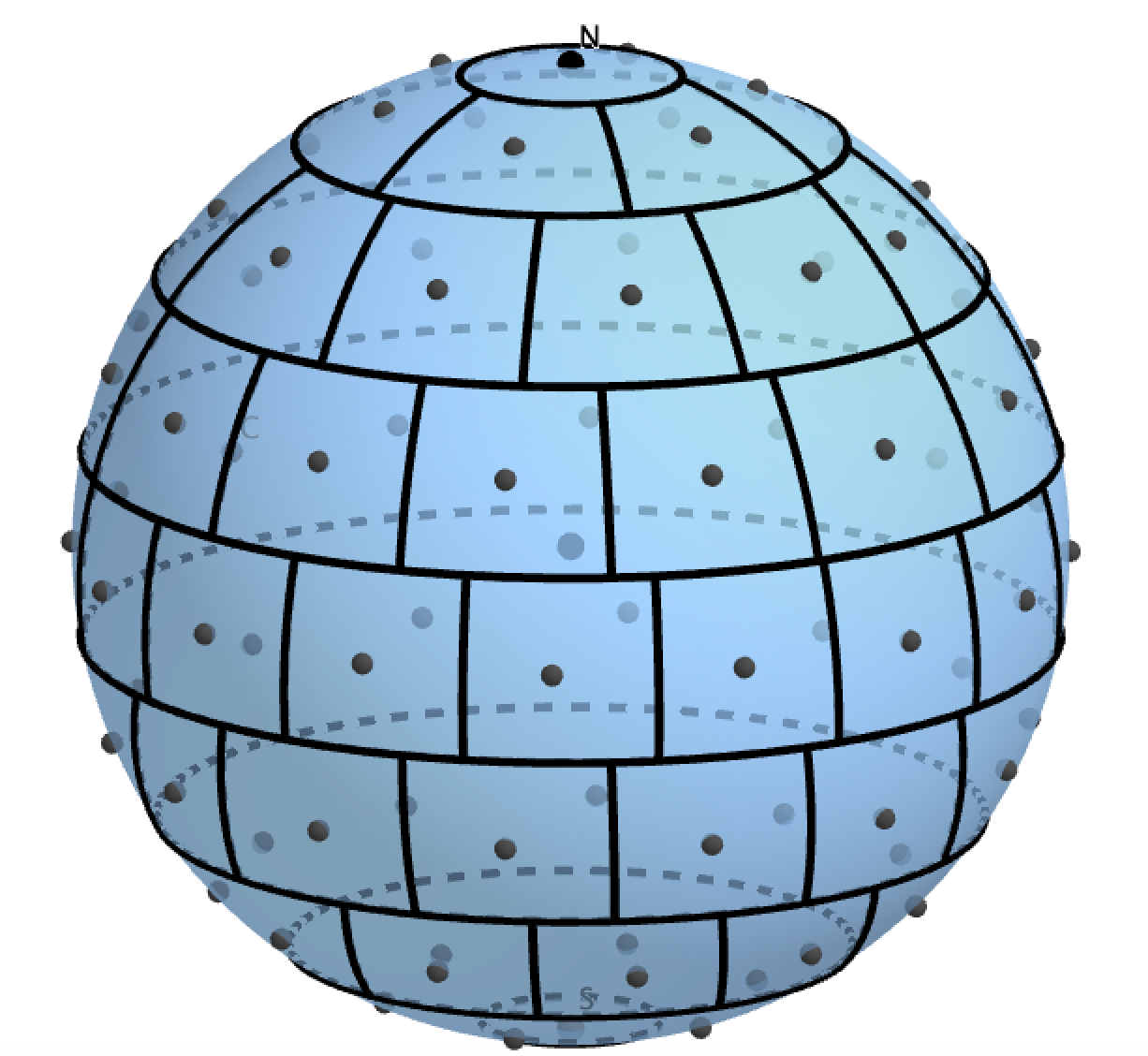}
\caption{Example of an area regular partition coming from the Diamond ensemble.}
\label{fig:figure1}
\end{figure}

\begin{defn}\label{defn_ARP}
Let $p,M$ be two positive integers with $p=2M-1$ and let us consider the following subsets of $\mathbb{S}^{2}$:
\begin{itemize}
\item A spherical cap centered at the North pole with height $1 -h_{1}$, that we denote by $R_{N}$.
\item The spherical rectangles on the North hemisphere given in spherical coordinates on the sphere by: 
	\begin{equation*}
	R_{j}^{i} = \left[ \frac{2\pi i}{r_{j}} + \frac{\pi}{r_{j}} + \theta_{j}, \frac{2\pi (i+1)}{r_{j}} + \frac{\pi}{r_{j}} + \theta_{j}\right]
	\times 
	\left[\arccos\left( h_{j} \right), 
	\arccos\left( h_{j+1} \right) \right].
	\end{equation*}
\item The spherical rectangles:
	\begin{equation*}
	R_{M}^{i} = \left[ \frac{2\pi i}{r_{M}} + \frac{\pi}{r_{M}} + \theta_{M}, \frac{2\pi (i+1)}{r_{M}} + \frac{\pi}{r_{M}} + \theta_{M}\right]
	\times 
	\left[\arccos\left( h_{M}\right),
	\pi-\arccos\left( h_{M}\right) \right],
	\end{equation*}
on a collar containing the equator.
\item The symmetrization of the North hemisphere.
\end{itemize}
Let $h_{j}$ be defined by the following recurrence relation:
\begin{equation*}
h_{1} = 1 - \frac{2}{N},
\qquad
h_{j+1} = h_{j} - \frac{2r_{j}}{N}
\end{equation*}
for $1 \leq j \leq M$ or in an explicit formula, using the notation from \eqref{eq_N_parcial}, by
\begin{equation*}
h_{j} = 1 - \frac{2N_{j}}{N}.
\end{equation*}

\end{defn}



\begin{prop}\label{thm_ARP}
The partition defined in Definition \ref{defn_ARP} is an area regular partition.
\end{prop}


\begin{proof}
From the definition, we can see that it is enough to compute the area of the North pole region, one of the regions $R_{j}^{i}$ and one of the regions $R_{M}^{i}$.
We start by computing the area of the region containing the North pole:

\begin{equation*}
A(R_{N})
=
2\pi (1 - h_{1})
=
2\pi \left( 1 - 1 + \frac{2}{N}  \right)
=
\frac{4\pi}{N}.
\end{equation*}
Now we consider a rectangle $R_{j}^{i}$ and compute its area.
\begin{multline*}
A(R_{j}^{i})
=
\int_{\frac{2\pi i}{r_{j}} + \frac{\pi}{r_{j}} + \theta_{j}}^{\frac{2\pi (i+1)}{r_{j}} + \frac{\pi}{r_{j}} + \theta_{j}}
\int_{\arccos\left( h_{j} \right)}^{\arccos\left( h_{j+1}\right)}
\sin (\theta) d\theta d\phi
\\
=
\left(
\int_{\frac{2\pi i}{r_{j}} + \frac{\pi}{r_{j}} + \theta_{j}}^{\frac{2\pi (i+1)}{r_{j}} + \frac{\pi}{r_{j}} + \theta_{j}}  d\phi
\right)
\left(
\int_{\arccos\left( h_{j} \right)}^{\arccos\left( h_{j+1}\right)}
\sin (\theta) d\theta
\right)
\\
=
\frac{2\pi}{r_{j}}
\left(
\cos \left( \arccos(h_{j}) \right)
- \cos \left( \arccos(h_{j+1}) \right)
\right)
=
\frac{2\pi}{r_{j}}
\left(
h_{j} -h_{j+1}
\right)
.
\end{multline*}
By the recurrence relation defining $h_{j+1}$, we have:
\begin{equation*}
A(R_{j}^{i})
=
\frac{2\pi}{r_{j}}
\left(
h_{j} -h_{j} + \frac{2r_{j}}{N}
\right)
=
\frac{4\pi}{N}
.
\end{equation*}
It only rest to see the case $R_{M}^{i}$. 
We compute the area of half of the region:
\begin{equation*}
\begin{split}
\frac{A \left( R_{M}^{i} \right)}{2}
=
&
\int_{\frac{2\pi i}{r_{M}} + \frac{\pi}{r_{M}} + \theta_{M}}^{\frac{2\pi (i+1)}{r_{M}} + \frac{\pi}{r_{M}} + \theta_{M}}
\int_{\arccos\left( h_{M}\right)}^{\frac{\pi}{2}}
\sin (\theta) d\theta d\phi
=
\frac{2\pi}{r_{M}}
h_{M}.
\end{split}
\end{equation*}
Using the explicit definition of $h_{M}$ we can write
\begin{multline*}
\frac{A \left( R_{M}^{i} \right)}{2}
=
\frac{2\pi}{r_{M}} \left(  1 - \frac{2}{N} N_{M} \right)
\\
=
\frac{2\pi}{N} \left( \frac{N}{r_{M}} - \frac{2N_{M}}{r_{M}} \right)
=
\frac{2\pi}{N}\frac{N - 2N_{M}}{r_{M}} 
=
\frac{2\pi}{N}\frac{r_{M}}{r_{M}} 
=
\frac{2\pi}{N} .
\end{multline*}

\end{proof}


\begin{prop}
Every region of the partition defined in Definition \ref{defn_ARP} contains a unique point of the Diamond ensemble, so it drops the notation of the regions.
\end{prop}

\begin{proof}
Since given a collar, it is partitioned in such a way that every point belongs to a different region, it is enough to prove that 
\begin{equation*}
h_{j+1}
<
z_{j}
<
h_{j}
\end{equation*}
for all $1\leq j \leq M-1$ and that $h_{M} >0$.

We start by proving that $z_{j} < h_{j}$, which follows easily from these two facts:
\begin{equation*}
\frac{2}{N-1}N_{j} 
>
\frac{2}{N}N_{j} 
\Rightarrow
1-\frac{2}{N-1}N_{j} 
<
1-\frac{2}{N}N_{j},
\end{equation*}
\begin{equation*}
\frac{r_{j} -1 }{N-1} \geq 0.
\end{equation*}
If we take now the characterization for $z_{j}$ and $h_{j}$ given in equation \eqref{zjs} and Definition \ref{defn_ARP} respectively, the proof is done.
To prove that $h_{j+1} \leq z_{j}$ we use the fact that $N > 2N_{j+1}$ for all $1\leq j \leq M-1$ and $r_{j} \geq r_{1} \geq 1$ by the properties of Definition \ref{defn_Diamond} so, 
\begin{equation*}
\begin{split}
Nr_{j} > 2N_{j+1}
&
\Rightarrow
2Nr_{j} > N(r_{j}-1) + 2N_{j+1}
\\
&
\iff
2N(N_{j+1}-N_{j}) >N(r_{j}-1) + 2N_{j+1}
\\
&
\iff
2NN_{j+1} - 2N_{j+1} > 2NN_{j}  +N(r_{j}-1)
\\
&
\iff
2N_{j+1}(N-1) > 2NN_{j}  +N(r_{j}-1)
\\
&
\iff
\frac{2N_{j+1}}{N} > \frac{2N_{j}}{N-1}  +\frac{r_{j}-1}{N-1}
\\
&
\iff
1-\frac{2N_{j+1}}{N} < 1- \frac{2N_{j}}{N-1}  -\frac{r_{j}-1}{N-1}
\\
&
\iff
 h_{j+1} < z_{j}.
\end{split}
\end{equation*}
We end with 
\begin{equation*}
h_{M}
=
1 - \frac{2N_{M}}{N}
=
1 -  \frac{N-r_{M}}{N}
=
\frac{r_{M}}{N}
>
0.
\end{equation*}
\end{proof}

We describe some properties of the area regular partition.

\begin{prop}
The radius of the region $R_{N}$ is $2 \arcsin \left(\frac{1}{\sqrt{N}}\right) \approx \frac{2}{\sqrt{N}}$ for big $N$.
\end{prop}

\begin{proof}
The proof consists on some trigonometric computations and is left to the reader.
\end{proof}

\begin{prop}\label{prop_bounded_regions}
For every rectangular region $R_{j}^{i}$, the length of the horizontal sides (those parallels to the equator) is bounded by 
\begin{equation*}
\frac{d_{1}}{\sqrt{N}}
<
\text{length of the horizontal sides of } R_{j}^{i}<
\frac{d_{2}}{\sqrt{N}},
\end{equation*}
where  $d_{1},d_{2} \in \mathbb{R}_{+}$ are fix constants depending only on the choice of parameters $n, t_{\ell}, \alpha_{\ell}, \beta_{\ell}$, $1\leq \ell \leq n$.
\end{prop}

\begin{proof}
By the symmetry of the model, we only work with the regions of the North hemisphere and the equator $R_{M}^{i}$.
Note that for every rectangular region of the North hemisphere, the side parallel to the equator that is closer to the North pole is shorter than the one that is closer to the equator, see Figure \ref{fig:figure1}.
In the case $R_{M}^{i}$ they are equal.
So it is enough to prove that
\begin{equation*}
\frac{d_{1}}{\sqrt{N}}
<
\frac{2\pi \sqrt{1 - h_{j}^{2}}}{r_{j}}
<
\frac{d_{2}}{\sqrt{N}}
\end{equation*}
for $1 \leq j \leq M$.
We develope this expression:
\begin{equation*}
\frac{2\pi \sqrt{1 - h_{j}^{2}}}{r_{j}}
>
\frac{d_{1}}{\sqrt{N}}
\iff
N(1 - h_{j}^{2})
>
\frac{d_{1}^2}{4\pi^2} r_{j}^{2}
\iff
N_{j}\left(1 - \frac{N_{j}}{N}\right)
>
\frac{d_{1}^2}{16\pi^2} r_{j}^{2}.
\end{equation*}
Since $1 \leq j \leq M$, we have that
\begin{equation*}
\frac{1}{2}
<
1 - \frac{N_{j}}{N}
<
1
\end{equation*}
Appliying Proposition \ref{lem_Nj} we have
\begin{equation*}
N_{j}\left(1 - \frac{N_{j}}{N}\right)
>
\frac{N_{j}}{2}
\geq
\frac{k_{1}}{2} r_{j}^2
\end{equation*}
and so it is enough to take $d_{1} = 2\sqrt{2}\pi\sqrt{k_{1}}$.
On the other hand, 
\begin{equation*}
N_{j}\left(1 - \frac{N_{j}}{N}\right)
<
N_{j}
\leq
k_{2} r_{j}^2,
\end{equation*}
and so we take $d_{2}=4\pi\sqrt{k_{2}}$.
\end{proof}

\begin{cor}\label{well_separated}
The heights of the rectangles $R_{j}^{i}$ for $1\leq j\leq M$ of the area regular partition  are also bounded by
\begin{equation*}
\frac{e_{1}}{\sqrt{N}}
<
\text{length of the vertical sides of } R_{j}^{i}<
\frac{e_{2}}{\sqrt{N}},
\end{equation*}
where $e_{1},e_{2}\in\mathbb{R}_{+}$ depend only on the choice of parameters $n, t_{\ell}, \alpha_{\ell}, \beta_{\ell}$, $1\leq \ell \leq n$.
\end{cor}


\begin{cor}\label{cor_diam}
The diameters of the rectangles $R_{j}^{i}$ for $1\leq j\leq M$ of the area regular partition are bounded by
\begin{equation*}
\frac{g_{1}}{\sqrt{N}}
<
\text{diam} \left( R_{j}^{i} \right)
<
\frac{g_{2}}{\sqrt{N}},
\end{equation*}
where $g_{1}, g_{2}\in \mathbb{R}_{+}$ depend only on the choice of parameters $n, t_{\ell}, \alpha_{\ell}, \beta_{\ell}$, $1\leq \ell \leq n$.
\end{cor}


Corollary \ref{cor_diam} implies that the mesh norm of the Diamond ensemble is bounded by $\frac{g_{2}}{\sqrt{N}}$.
So in particular we can state that the Diamond ensemble is a good covering.


\subsection{A concrete example}\label{sub_concrete}

We consider in this section the simple model defined in \cite[Section 4.1]{EB18_3} and compute explicitly all the constants presented in the previous section.
Following the notation from Definition \ref{defn_Diamond}, we choose $n=1$ and $r_j=4j$ for $1\leq j\leq M$. 
Then, for all $j\in\{1,\ldots,M\}$ we have
\[
z_j=1-\frac{1+4j^2}{N-1}.
\] 
The number of parallels is $2M-1$ and the number of points is 
\begin{equation*}
N
=
2 -4M+ 2\sum_{j=1}^{M} 4j
=
2 + 4M^2.
\end{equation*}
\begin{equation*}
N_{j}
=
1 + \sum_{k=1}^{j-1} 4k
=
2j^2 - 2j + 1
.
\end{equation*}
We consider the partition of $\mathbb{S}^{2}$ defined in Definition \ref{defn_ARP} 
where
	\begin{equation*}
	h_{j}
	=
	1 - \frac{2}{N} - \frac{4j(j-1)}{N}
	=
	\frac{-4}{N} j^{2} + \frac{4}{N} j + \left( 1 - \frac{2}{N} \right),
	\end{equation*}
for $1\leq j \leq M$ and is given by the recurrence relation:
	\begin{equation*}
	h_{j+1}
	=
	h_{j} - \frac{8j}{N}.
	\end{equation*}
We can write again Proposition \ref{prop_bounded_regions} with explicit constants $d_{1}$ and $d_{2}$.

\begin{prop}\label{prop_bound_concrete}
For every rectangular region $R_{j}^{i}$ from the area regular partition described above, the length of the horizontal sides (those parallels to the equator) is bounded by 
\begin{equation*}
\frac{\pi}{\sqrt{2}}\frac{1}{\sqrt{N}} 
<
\text{length of the horizontal sides of } R_{j}^{i}
<
\frac{\pi\sqrt{2}}{\sqrt{N}}.
\end{equation*}
\end{prop}

\begin{proof}
As in proof of Proposition \ref{prop_bounded_regions}, we consider the quantity
\begin{equation*}
\frac{2\pi \sqrt{1 - h_{j}^{2}}}{r_{j}}
=
\frac{2\pi \sqrt{\frac{4N_{j}}{N}\left( 1 - \frac{N_{j}}{N}\right)}}{4j}
=
\frac{\pi }{N}
\sqrt{\frac{\left( 2j^2 - 2j + 1 \right)\left( 4M^2 - 2j^2 + 2j + 1 \right)}{j^2}}.
\end{equation*}
First we bound
\begin{equation*}
1
\leq
\frac{2j^2 - 2j + 1}{j^2}
<
2
\end{equation*}
for all $1\leq j \leq M$.
On the other hand,
\begin{equation*}
2M^2 + 2M + 1
\leq
4M^2 - 2j^2 + 2j + 1
\leq 
4M^2  + 1
\end{equation*}
for all $1\leq j \leq M$.
Since all quantities are positive, we have
\begin{equation*}
\sqrt{2M^2 + 2M + 1}
\leq
\sqrt{(4M^2 - 2j^2 + 2j + 1)\left( \frac{2j^2 - 2j + 1}{j^2}
 \right)}
<
\sqrt{2(4M^2  + 1)}.
\end{equation*}
We rewrite the expressions in terms of $N$
\begin{equation*}
\frac{\pi}{N}\sqrt{2M^2 + 2M + 1}
=
\frac{\pi}{N}\sqrt{\frac{N}{2} + \sqrt{N-2}}
\geq
\frac{\pi}{\sqrt{2}}\frac{1}{\sqrt{N}} 
\end{equation*}
and
\begin{equation*}
\frac{\pi}{N}\sqrt{2(4M^2  + 1)}
=
\frac{\pi}{N}\sqrt{2N}
=
\frac{\pi\sqrt{2}}{\sqrt{N}}.
\end{equation*}

\end{proof}

We can easily deduce bounds for the other quantities for this model as in Corolaries \ref{well_separated} and \ref{cor_diam}.




\section{Proof of Theorem \ref{thm_menor_que}}\label{sec_Proof_menor_que}

As we mentioned before, to prove Theorem \ref{thm_menor_que} we follow the general lines of the proof proposed in \cite[Theorem 24D]{87e609866d654db2aa5e674048289af3}.
\medskip

Given a family of points coming from the Diamond ensemble for some choice of parameters $n,t_{\ell},\alpha_{\ell},\beta_{\ell}$ for $1\leq \ell \leq n$, we consider the associated area regular partition given in Definition \ref{defn_ARP}.
Let us take a spherical cap on the sphere $\mathbb{S}^2$ and denote it by $C$.
We can split 
\begin{equation*}
C = \tilde{C} \cup \dot{C}
\end{equation*}
where $\dot{C}$ is the union of all the regions of the area regular partition that are completelly contained in $C$.
Therefore, $\tilde{C}$ is the union of all the regions of the area regular partition that are partially contained in $C$ intersected with $C$.
Then we have:
\begin{multline*}
 D_{\text{sup},\text{cap}} (\diamond\left( N \right)) 
=
\sup\limits_{C \in \text{cap}}
\left|
\frac{\diamond\left( N \right) \cap C}{N}
-
\frac{\mu(C)}{4\pi}
\right|
\\
=
\sup\limits_{C \in \text{cap}}
\left|
\frac{\diamond\left( N \right) \cap \tilde{C}}{N}
+
\frac{\diamond\left( N \right) \cap \dot{C}}{N}
-
\frac{\mu(\tilde{C})}{4\pi}
-
\frac{\mu(\dot{C})}{4\pi}
\right|.
\end{multline*}
Since we are taking an area regular partition, we have
\begin{equation*}
\frac{\diamond\left( N \right) \cap \dot{C}}{N}
=
\frac{\mu(\dot{C})}{4\pi}
\end{equation*}
and so,
\begin{equation*}
 D_{\text{sup},\text{cap}} (\diamond\left( N \right)) 
=
\sup\limits_{C \in \text{cap}}
\left|
\frac{\diamond\left( N \right) \cap \tilde{C}}{N}
-
\frac{\mu(\tilde{C})}{4\pi}
\right|.
\end{equation*}
Now let us prove that the border of any spherical cap $C$ pass through at most $k\sqrt{N}$ different regions of our partition, with $k\in\mathbb{R}_{+}$ depending only on the choice of parameters $n, t_{\ell}, \alpha_{\ell}, \beta_{\ell}$, $1\leq \ell \leq n$.
In order to do so, we consider the intersection of the border of our spherical cap, that we wil denote by $\mathcal{C}$ and a collar $Z_{j} = \displaystyle\cup_{i=1}^{r_{j}} R_{j}^{i}$.
\begin{equation*}
\mathcal{L}_{j}
=
Z_{j}\cap \mathcal{C}, 
\end{equation*}
and we consider the lenght of $\mathcal{L}_{j}$, that we denote by $|\mathcal{L}_{j}|$.
Note that $\mathcal{C}$ can pass through each $Z_{j}$ at most twice non consecutive times, see Figure \ref{fig:figure2}.
Then the number of regions that $\mathcal{L}_{j}$ pass through, that we denote by $N(\mathcal{L}_{j})$, is bounded by:
\begin{equation*}
N(\mathcal{L}_{j})
\leq
4+
\frac{|\mathcal{L}_{j}|}{\frac{d_{1}}{\sqrt{N}}}
\end{equation*}
with $d_{1}$ as in Proposition \ref{prop_bounded_regions}.
So, the number of regions that the border of $C$ pass through is bounded by
\begin{multline*}
\sum_{j=1}^{2M-1}
N(\mathcal{L}_{j})
\leq
\sum_{j=1}^{2M-1}
\left(
4+
\frac{|\mathcal{L}_{j}|}{\frac{d_{1}}{\sqrt{N}}}
\right)
=
4(2M-1) + 
\frac{\sqrt{N}}{d_{1}}
\sum_{j=1}^{2M-1}
|\mathcal{L}_{j}| 
\\
\leq
4(2M-1) + 
\frac{2\pi}{d_{1}}\sqrt{N}
\leq
\frac{8}{\sqrt{a_{1}}}\sqrt{N} - 4 + 
\frac{2\pi}{d_{1}}\sqrt{N}
\leq
\left(
\frac{8}{\sqrt{a_{1}}}
+
\frac{2\pi}{d_{1}}
\right)
\sqrt{N},
\end{multline*}
where we have used Lemma \ref{pos} to bound $M$.
\begin{figure}[htp]
\centering
\includegraphics[width=.33\textwidth]{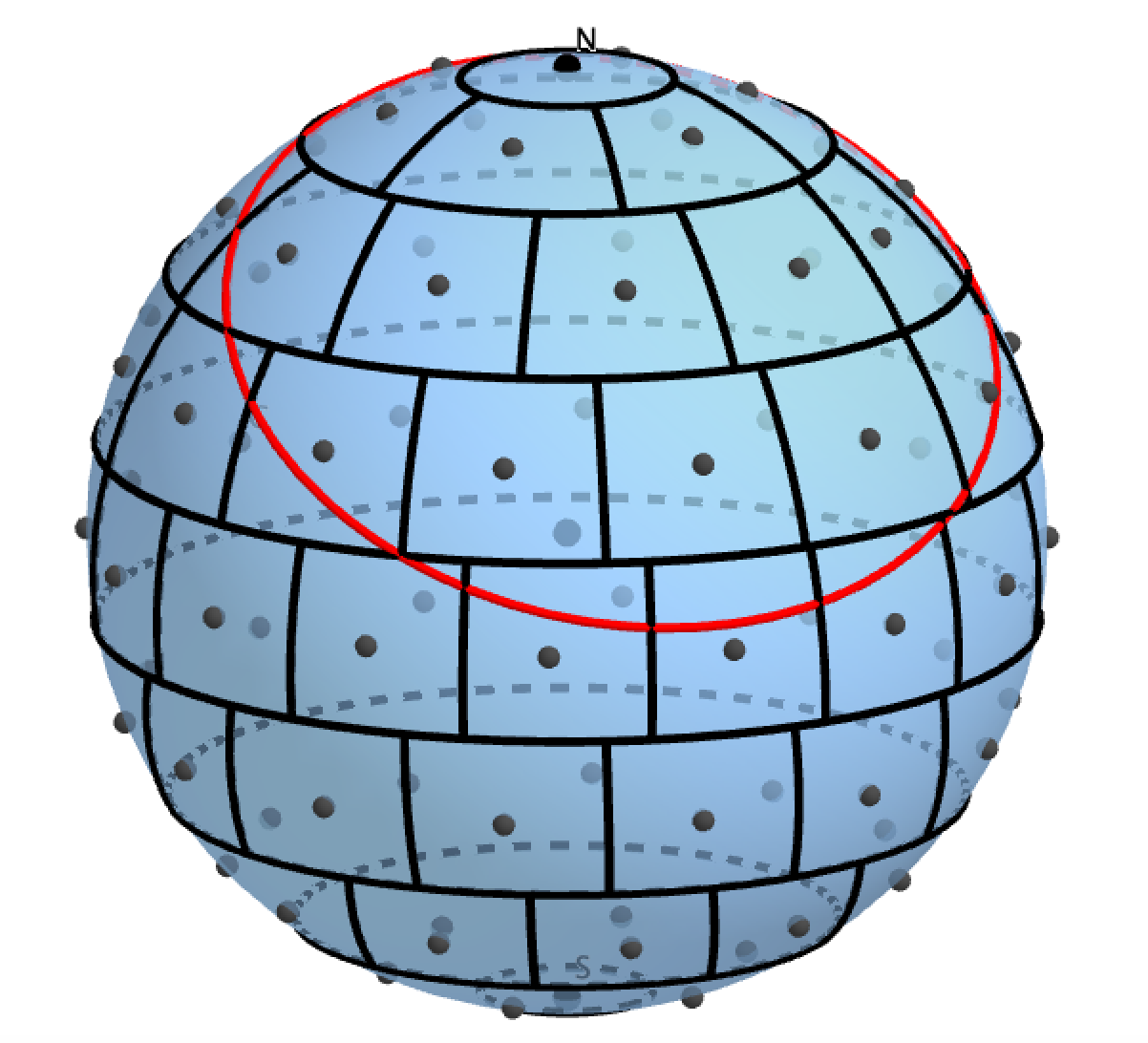}\hfill
\includegraphics[width=.32\textwidth]{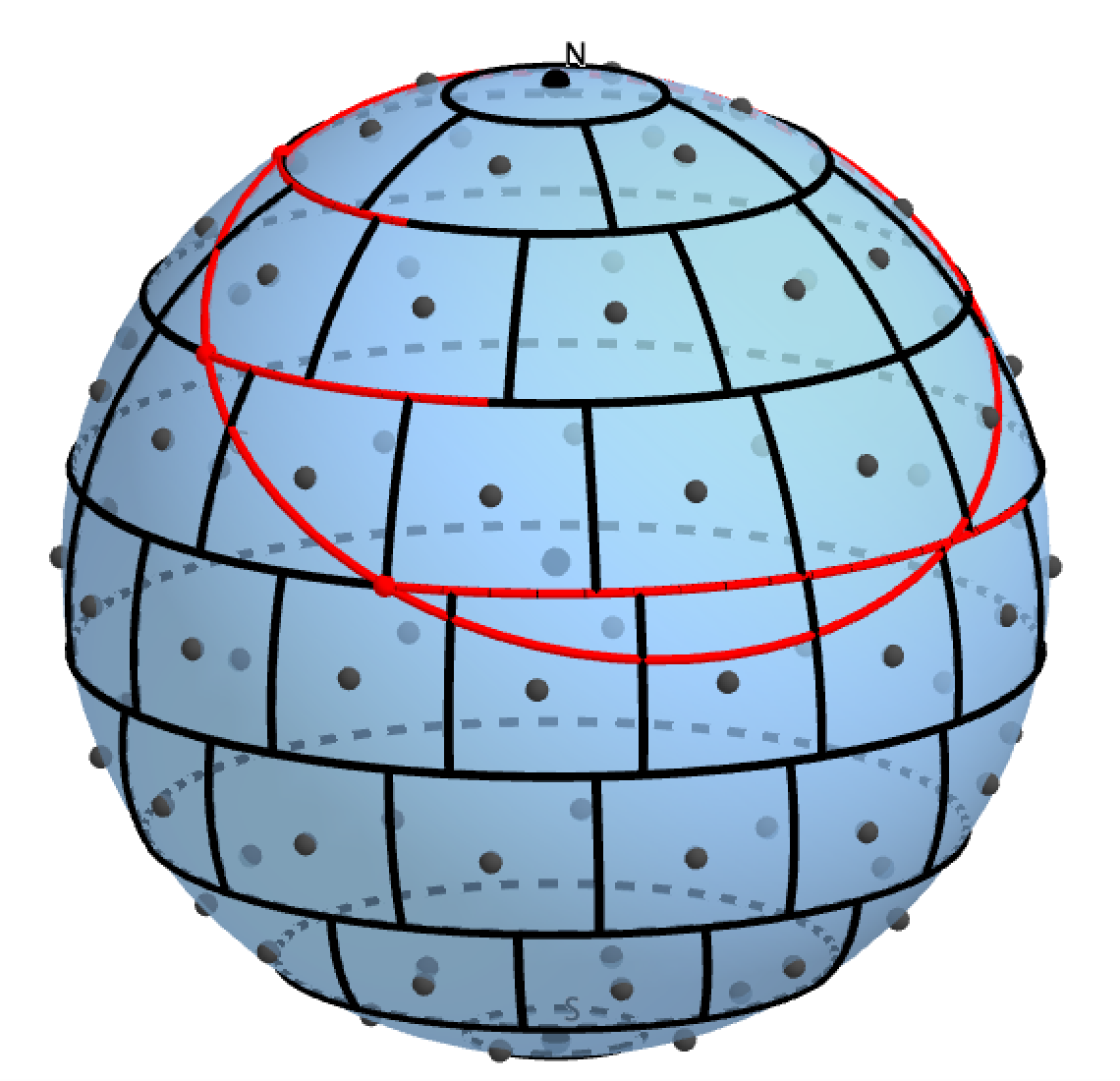}\hfill
\includegraphics[width=.33\textwidth]{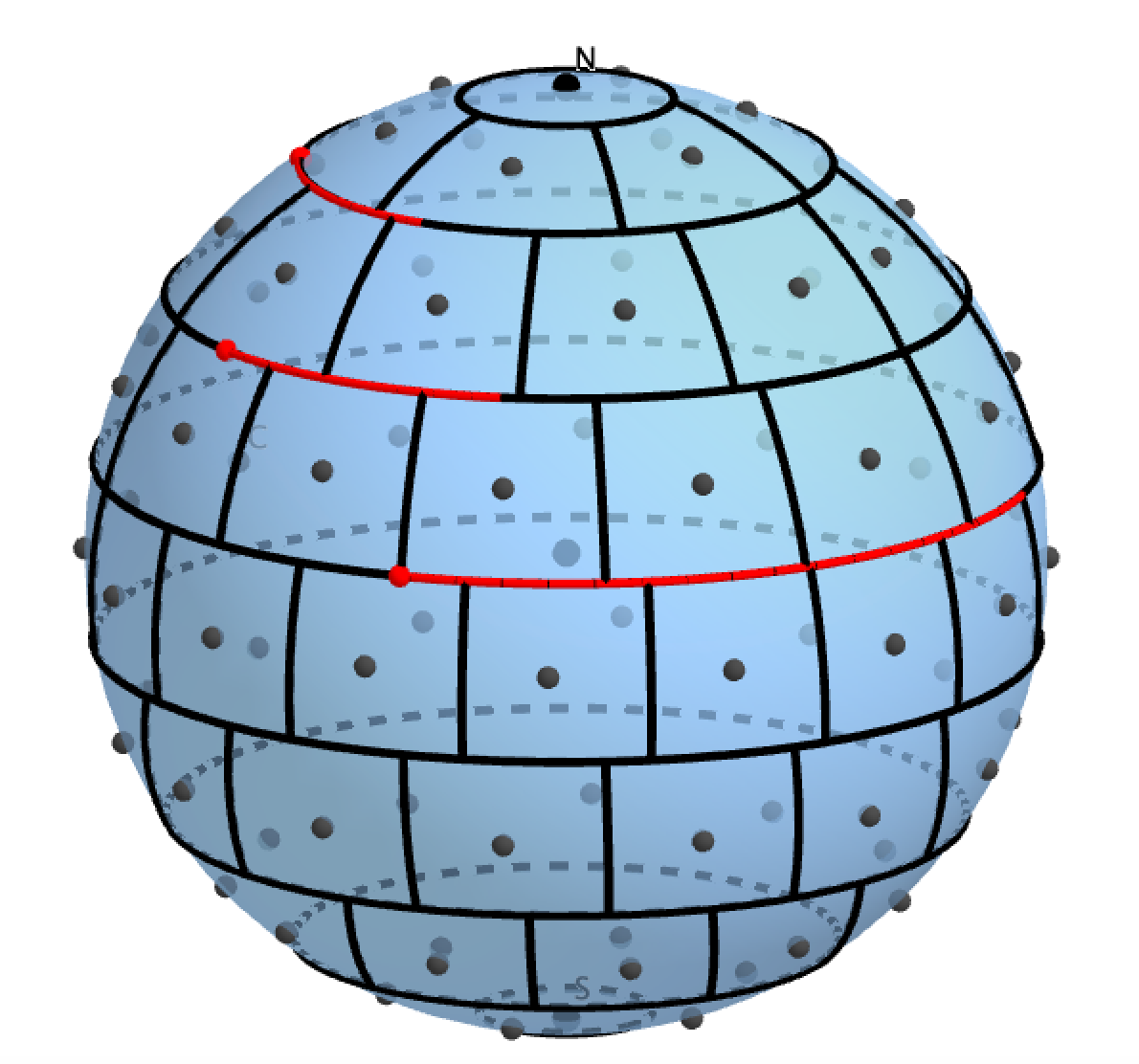}
\caption{Decomposition of the border of the spherical cap.}
\label{fig:figure2}
\end{figure}
Since every region has area $\frac{4\pi}{N}$ we conclude that
\begin{equation*}
 D_{\text{sup},\text{cap}} (\diamond\left( N \right) 
=
\sup\limits_{C \in \text{cap}}
\left|
\frac{\diamond\left( N \right) \cap \tilde{C}}{N}
-
\frac{\mu(\tilde{C})}{4\pi}
\right|
\leq
\left(
\frac{8}{\sqrt{a_{1}}}
+
\frac{2\pi}{d_{1}}
\right)
\frac{1}{\sqrt{N}}.
\end{equation*}

Note that we are not taking into account the regions containing the North or the South pole since they are meaningless for the asymptotics.


\section{Proof of Theorem \ref{thm_mayor_que}}\label{sec_Proof_mayor_que}

For proving Theorem \ref{thm_mayor_que} we consider the very specific spherical cap consisting on the upper half semisphere containing the line of the equator.
Then the expression
\begin{equation*}
\left|
\frac{\diamond\left( N \right) \cap C}{N}
-
\frac{\mu(C)}{4\pi}
\right|
\end{equation*}
can be simplified.
For the simetry of the model, 
\begin{equation*}
\diamond\left( N \right) \cap C
=
\frac{N}{2} + \frac{r_{M}}{2}
\end{equation*}
where by $r_{M}$ we denote the number of points that lie in the equator and $\frac{\mu(C)}{4\pi} = \frac{1}{2}$
Then we have
\begin{equation*}
\left|
\frac{\diamond\left( N \right) \cap C}{N}
-
\frac{\mu(C)}{4\pi}
\right|
=
\left|
\frac{\frac{N}{2} + \frac{r_{M}}{2}}{N}
-
\frac{1}{2}
\right|
=
\left|
\frac{1}{2}
+
\frac{r_{M}}{2N}
-
\frac{1}{2}
\right|
=
\frac{r_{M}}{2N}.
\end{equation*}
From Definition \ref{defn_Diamond} we know that
\begin{equation*}
r_{M} 
\geq
r_{t_{1}}
\geq
cM,
\end{equation*}
and from Lemma \ref{pos} we have that
\begin{equation*}
N
\leq 
a_{2}M^2.
\end{equation*}
So, we conclude:
\begin{equation*}
\left|
\frac{\diamond\left( N \right) \cap C}{N}
-
\frac{\mu(C)}{4\pi}
\right|
=
\frac{r_{M}}{2N}
\geq
\frac{cM}{2N}
\geq
\frac{c\sqrt{N}}{2\sqrt{a_{2}}N}
=
\frac{c}{2\sqrt{a_{2}}}\frac{1}{\sqrt{N}}.
\end{equation*}
Then, it is enough to take $c_{1} = \frac{c}{2\sqrt{a_{2}}}$ to conclude that
\begin{equation*}
 D_{\text{sup},\text{cap}} (\diamond\left( N \right)) 
=
\sup\limits_{C \in \text{cap}}
\left|
\frac{\diamond\left( N \right) \cap C}{N}
-
\frac{\mu(C)}{4\pi}
\right|
\geq
\frac{c_{1}}{\sqrt{N}}.
\end{equation*}


\section{Proof of Theorem \ref{thm_concrete}}\label{sec_Proof_concrete}

As for Theorem \ref{thm_main}, we split the proof on Theorem \ref{thm_concrete} into two lemmas.

\begin{lem}\label{lem1}
Let $\diamond (N)$ be the Diamond ensemble defined by $n=1$ and $r_j=4j$ for $1\leq j\leq M$. 
Then
\begin{equation*}
D_{\text{sup},\text{cap}}\left(\diamond (N)\right)
<
\frac{
4+ 2\sqrt{2}
}{\sqrt{N}}
.
\end{equation*}
\end{lem}

\begin{proof}
We follow the proof from Theorem \ref{thm_menor_que} and the bounds given in Proposition \ref{prop_bound_concrete}. 
Then we have the following bound
\begin{multline*}
\sum_{j=1}^{2M-1}
N(\mathcal{L}_{j})
<
\sum_{j=1}^{2M-1}
\left(
4+
\frac{|\mathcal{L}_{j}|}{\frac{\pi}{\sqrt{2}}\frac{1}{\sqrt{N}} }
\right)
=
4(2M-1) + 
\frac{\sqrt{2}}{\pi}\sqrt{N} 
\sum_{j=1}^{2M-1}
|\mathcal{L}_{j}| 
\\
\leq
4\sqrt{N-2}-4 + 
2\sqrt{2}\sqrt{N}
<
\left(
4+ 2\sqrt{2}
\right)\sqrt{N}
.
\end{multline*}
Then, we have
\begin{equation*}
 D_{\text{sup},\text{cap}} (\diamond\left( N \right)) 
=
\sup\limits_{C \in \text{cap}}
\left|
\frac{\diamond\left( N \right) \cap \tilde{C}}{N}
-
\frac{\mu(\tilde{C})}{4\pi}
\right|
\leq
\frac{
4+ 2\sqrt{2}
}{\sqrt{N}}
.
\end{equation*}

\end{proof}

\begin{lem}\label{lem2}
Let $\diamond (N)$ be the Diamond ensemble defined by $n=1$ and $r_j=4j$ for $1\leq j\leq M$. 
Then
\begin{equation*}
D_{\text{sup},\text{cap}}\left(\diamond (N)\right)
\geq
\frac{1}{\sqrt{N}}+ o\left( \frac{1}{\sqrt{N}} \right).
\end{equation*}
\end{lem}

\begin{proof}
We are going to consider a subfamily of spherical caps in $\mathbb{S}^2$ formed by the caps that are centered at the North pole and whose border is one of the parallels where we have chosen the points, i.e. one of the parallels defined by the $z_{j}$'s.
For the symmetry of the model, it is enough to consider $1\leq j \leq M$.
The discrepancy for these particular caps reads
\begin{multline*}
\sup\limits_{1\leq j \leq M}
\left|
\frac{\diamond\left( N \right) \cap C}{N}
-
\frac{\mu(C)}{4\pi}
\right|
=
\sup\limits_{1\leq j \leq M}
\left|
\frac{N_{j+1}}{N}
-
\frac{2\pi(1-z_{j})}{4\pi}
\right|
\\
=
\sup\limits_{1\leq j \leq M}
\left|
\frac{N-2 - 4j^2 + 4(N-1)j}{2N(N-1)}
\right|,
\end{multline*}
where $N-2 - 4j^2 + 4(N-1)j>0$ for all $1\leq j \leq M$, and $f(x) = N-2 - 4x^2 + 4(N-1)x$ is an increasing function in the interval $[1,M]$, so
\begin{multline*}
\sup\limits_{1\leq j \leq M}
\left|
\frac{\diamond\left( N \right) \cap C}{N}
-
\frac{\mu(C)}{4\pi}
\right|
\\
=
\frac{N-2 - 4M^2 + 4(N-1)M}{2N(N-1)}
=
\frac{\sqrt{N-2}}{N}
=
\frac{1}{\sqrt{N}} + o\left( \frac{1}{\sqrt{N}} \right).
\end{multline*}

\end{proof}


\subsection*{Acknowledgements}

I would like to thank Peter Grabner for our discussions on the topic and for introducing me to the book \cite{87e609866d654db2aa5e674048289af3}, it was such a nice lecture.



\end{document}